\documentclass[12pt,reqno,draft]{amsart}
\usepackage{amssymb,amscd,url}

\begin{document}

\allowdisplaybreaks
\hyphenation{non-archi-me-dean}
\newif\ifdraft \drafttrue
\draftfalse
\newcommand{\DRAFTNUMBER}{1}
\newcommand{\DATE}{\today}
\newcommand{\TITLE}{Examples of dynamical degree equals arithmetic degree}
\newcommand{\TITLERUNNING}{Dynamical degree equals arithmetic degree}



\newtheorem{theorem}{Theorem}
\newtheorem{lemma}[theorem]{Lemma}
\newtheorem{conjecture}[theorem]{Conjecture}
\newtheorem{proposition}[theorem]{Proposition}
\newtheorem{corollary}[theorem]{Corollary}
\newtheorem*{claim}{Claim}

\theoremstyle{definition}
\newtheorem*{definition}{Definition}
\newtheorem{remark}[theorem]{Remark}
\newtheorem{example}[theorem]{Example}
\newtheorem{question}[theorem]{Question}

\theoremstyle{remark}
\newtheorem*{acknowledgement}{Acknowledgements}

\newcommand\CaseAlt[2]{\par\vspace{1ex minus 0.5ex}\noindent%
\framebox{\textbf{Case #1}.\enspace\emph{#2}}\par\noindent\ignorespaces}

\newcommand\Case[1]{\par\vspace{1ex minus 0.5ex}\noindent%
\textbf{Case #1}:\enspace\ignorespaces}



\newenvironment{notation}[0]{%
  \begin{list}%
    {}%
    {\setlength{\itemindent}{0pt}
     \setlength{\labelwidth}{4\parindent}
     \setlength{\labelsep}{\parindent}
     \setlength{\leftmargin}{5\parindent}
     \setlength{\itemsep}{0pt}
     }%
   }%
  {\end{list}}

\newenvironment{parts}[0]{%
  \begin{list}{}%
    {\setlength{\itemindent}{0pt}
     \setlength{\labelwidth}{1.5\parindent}
     \setlength{\labelsep}{.5\parindent}
     \setlength{\leftmargin}{2\parindent}
     \setlength{\itemsep}{0pt}
     }%
   }%
  {\end{list}}
\newcommand{\Part}[1]{\item[\upshape#1]}

\renewcommand{\a}{\alpha}
\newcommand{\aupper}{\overline{\alpha}}
\newcommand{\alower}{\underline{\alpha}}
\renewcommand{\b}{\beta}
\newcommand{\bfbeta}{{\boldsymbol{\beta}}}
\newcommand{\g}{\gamma}
\renewcommand{\d}{\delta}
\newcommand{\e}{\epsilon}
\newcommand{\bfepsilon}{\boldsymbol{\epsilon}}
\newcommand{\f}{\varphi}
\newcommand{\bfphi}{{\boldsymbol{\f}}}
\renewcommand{\l}{\lambda}
\newcommand{\bfl}{{\boldsymbol{\lambda}}}
\renewcommand{\k}{\kappa}
\newcommand{\lhat}{\hat\lambda}
\newcommand{\m}{\mu}
\newcommand{\bfmu}{{\boldsymbol{\mu}}}
\renewcommand{\o}{\omega}
\renewcommand{\r}{\rho}
\newcommand{\rbar}{{\bar\rho}}
\newcommand{\s}{\sigma}
\newcommand{\sbar}{{\bar\sigma}}

\renewcommand{\t}{\tau}
\newcommand{\z}{\zeta}

\newcommand{\D}{\Delta}
\newcommand{\G}{\Gamma}
\newcommand{\F}{\Phi}
\renewcommand{\L}{\Lambda}
\newcommand{\ga}{{\mathfrak{a}}}
\newcommand{\gb}{{\mathfrak{b}}}
\newcommand{\gn}{{\mathfrak{n}}}
\newcommand{\gp}{{\mathfrak{p}}}
\newcommand{\gP}{{\mathfrak{P}}}
\newcommand{\gq}{{\mathfrak{q}}}

\newcommand{\Abar}{{\bar A}}
\newcommand{\Ebar}{{\bar E}}
\newcommand{\Kbar}{{\bar K}}
\newcommand{\Pbar}{{\bar P}}
\newcommand{\Sbar}{{\bar S}}
\newcommand{\Tbar}{{\bar T}}
\newcommand{\ybar}{{\bar y}}
\newcommand{\phibar}{{\bar\f}}

\newcommand{\ftilde}{{\tilde f}}

\newcommand{\Acal}{{\mathcal A}}
\newcommand{\Bcal}{{\mathcal B}}
\newcommand{\Ccal}{{\mathcal C}}
\newcommand{\Dcal}{{\mathcal D}}
\newcommand{\Ecal}{{\mathcal E}}
\newcommand{\Fcal}{{\mathcal F}}
\newcommand{\Gcal}{{\mathcal G}}
\newcommand{\Hcal}{{\mathcal H}}
\newcommand{\Ical}{{\mathcal I}}
\newcommand{\Jcal}{{\mathcal J}}
\newcommand{\Kcal}{{\mathcal K}}
\newcommand{\Lcal}{{\mathcal L}}
\newcommand{\Mcal}{{\mathcal M}}
\newcommand{\Ncal}{{\mathcal N}}
\newcommand{\Ocal}{{\mathcal O}}
\newcommand{\Pcal}{{\mathcal P}}
\newcommand{\Qcal}{{\mathcal Q}}
\newcommand{\Rcal}{{\mathcal R}}
\newcommand{\Scal}{{\mathcal S}}
\newcommand{\Tcal}{{\mathcal T}}
\newcommand{\Ucal}{{\mathcal U}}
\newcommand{\Vcal}{{\mathcal V}}
\newcommand{\Wcal}{{\mathcal W}}
\newcommand{\Xcal}{{\mathcal X}}
\newcommand{\Ycal}{{\mathcal Y}}
\newcommand{\Zcal}{{\mathcal Z}}

\renewcommand{\AA}{\mathbb{A}}
\newcommand{\BB}{\mathbb{B}}
\newcommand{\CC}{\mathbb{C}}
\newcommand{\FF}{\mathbb{F}}
\newcommand{\GG}{\mathbb{G}}
\newcommand{\NN}{\mathbb{N}}
\newcommand{\PP}{\mathbb{P}}
\newcommand{\QQ}{\mathbb{Q}}
\newcommand{\RR}{\mathbb{R}}
\newcommand{\ZZ}{\mathbb{Z}}

\newcommand{\bfa}{{\boldsymbol a}}
\newcommand{\bfb}{{\boldsymbol b}}
\newcommand{\bfc}{{\boldsymbol c}}
\newcommand{\bfe}{{\boldsymbol e}}
\newcommand{\bff}{{\boldsymbol f}}
\newcommand{\bfg}{{\boldsymbol g}}
\newcommand{\bfh}{{\boldsymbol h}}
\newcommand{\bfp}{{\boldsymbol p}}
\newcommand{\bfr}{{\boldsymbol r}}
\newcommand{\bfs}{{\boldsymbol s}}
\newcommand{\bft}{{\boldsymbol t}}
\newcommand{\bfu}{{\boldsymbol u}}
\newcommand{\bfv}{{\boldsymbol v}}
\newcommand{\bfw}{{\boldsymbol w}}
\newcommand{\bfx}{{\boldsymbol x}}
\newcommand{\bfy}{{\boldsymbol y}}
\newcommand{\bfz}{{\boldsymbol z}}
\newcommand{\bfA}{{\boldsymbol A}}
\newcommand{\bfC}{{\boldsymbol C}}
\newcommand{\bfF}{{\boldsymbol F}}
\newcommand{\bfB}{{\boldsymbol B}}
\newcommand{\bfD}{{\boldsymbol D}}
\newcommand{\bfE}{{\boldsymbol E}}
\newcommand{\bfG}{{\boldsymbol G}}
\newcommand{\bfI}{{\boldsymbol I}}
\newcommand{\bfM}{{\boldsymbol M}}
\newcommand{\bfzero}{{\boldsymbol{0}}}

\newcommand{\Amp}{\operatorname{Amp}}
\newcommand{\Aut}{\operatorname{Aut}}
\newcommand{\bad}{\textup{bad}}
\newcommand{\Disc}{\operatorname{Disc}}
\newcommand{\dist}{\Delta}  
\newcommand{\Div}{\operatorname{Div}}
\newcommand{\End}{\operatorname{End}}
\newcommand{\Eff}{\operatorname{Eff}}
\newcommand{\Family}{{\mathcal A}}  
\newcommand{\Fatou}{{\mathcal F}}
\newcommand{\Fbar}{{\bar{F}}}
\newcommand{\Fix}{\operatorname{Fix}}
\newcommand{\Gal}{\operatorname{Gal}}
\newcommand{\GL}{\operatorname{GL}}
\newcommand{\good}{\textup{good}}
\newcommand{\hplus}{h^{\scriptscriptstyle+}}
\newcommand{\Index}{\operatorname{Index}}
\newcommand{\Image}{\operatorname{Image}}
\newcommand{\interior}{\operatorname{int}}
\newcommand{\Julia}{{\mathcal J}}
\newcommand{\liftable}{{\textup{liftable}}}
\newcommand{\hhat}{{\hat h}}
\newcommand{\hhatlower}{\underline{\hat h}}
\newcommand{\hhatplus}{{\hat h^{\scriptscriptstyle+}}}
\newcommand{\hhatminus}{{\hat h^{\scriptscriptstyle-}}}
\newcommand{\hhatpm}{{\hat h^{\scriptscriptstyle\pm}}}
\newcommand{\Ker}{{\operatorname{ker}}}
\newcommand{\Lift}{\operatorname{Lift}}
\newcommand{\mG}{\hat{g}}        
\newcommand{\MOD}[1]{~(\textup{mod}~#1)}
\newcommand{\Nef}{\operatorname{Nef}}
\newcommand{\Norm}{{\operatorname{\mathsf{N}}}}
\newcommand{\notdivide}{\nmid}
\newcommand{\normalsubgroup}{\triangleleft}
\newcommand{\NS}{{\operatorname{NS}}}
\newcommand{\odd}{{\operatorname{odd}}}
\newcommand{\onto}{\twoheadrightarrow}
\newcommand{\ord}{\operatorname{ord}}
\newcommand{\Orbit}{{\mathcal O}}
\newcommand{\PGL}{\operatorname{PGL}}
\newcommand{\Pic}{\operatorname{Pic}}
\newcommand{\Prob}{\operatorname{Prob}}
\newcommand{\psef}{\textup{psef}}
\newcommand{\Qbar}{{\bar{\QQ}}}
\newcommand{\rank}{\operatorname{rank}}
\newcommand{\Res}{{\operatorname{Res}}}
\newcommand{\Resultant}{\operatorname{Res}}
\newcommand{\rest}[2]{\left.{#1}\right\vert_{{#2}}}  
\renewcommand{\setminus}{\smallsetminus}
\newcommand{\Span}{\operatorname{Span}}
\newcommand{\Spec}{{\operatorname{Spec}}}
\newcommand{\Supp}{\operatorname{Supp}}
\newcommand{\tors}{{\textup{tors}}}
\newcommand{\Trace}{\operatorname{Trace}}
\newcommand{\UHP}{{\mathfrak{h}}}    

\newcommand{\longhookrightarrow}{\lhook\joinrel\longrightarrow}
\newcommand{\longonto}{\relbar\joinrel\twoheadrightarrow}

\newcommand{\halfquad}{\hspace{.5em}}


\title[\TITLERUNNING]{\TITLE}
\date{\DATE \ifdraft{} --- Draft \DRAFTNUMBER\fi}

\author[Shu Kawaguchi and Joseph H. Silverman]
  {Shu Kawaguchi and Joseph H. Silverman}
\email{kawaguch@math.kyoto-u.ac.jp, jhs@math.brown.edu}
\address{Department of Mathematics, Faculty of Science, Kyoto University, 
Kyoto, 606-8502, Japan}
\address{Mathematics Department, Box 1917
         Brown University, Providence, RI 02912 USA}
\thanks{The first author's research supported by JSPS
grant-in-aid for young scientists (B) 24740015.
The second author's research supported by NSF DMS-0854755
and Simons Foundation grant \#241309.}

\begin{abstract}
Let $f:X\dashrightarrow X$ be a dominant rational map of a projective
variety defined over~$\Qbar$. An important geoemtric-dynamical
invariant of~$f$ is its (first) dynamical degree~$\d_f=\lim
\rho((f^n)^*)^{1/n}$.  For points $P\in X(\Qbar)$ whose forward orbits
are well-defined, there is an analogous arithmetic degree
$\aupper_f(P)=\limsup h_X\bigl(f^n(P)\bigr)^{1/n}$, where~$h_X$ is an ample
Weil height on~$X$.  In an earlier paper, we proved the fundamental
inequality $\aupper_f(P)\le\d_f$ and conjectured that $\aupper_f(P)=\d_f$
whenever the orbit of~$P$ is Zariski dense. In this paper we show that
the conjecture is true for several types of maps.  In other cases, we
provide support for the conjecture by proving that there is a Zariski
dense set of points with disjoint orbits and satisfying
$\aupper_f(P)=\d_f$.
\end{abstract}


\maketitle

\section*{Introduction}
\label{section:intro}

Let $X/\CC$ be a projective variety, let $f:X\dashrightarrow X$ be a
dominant rational map, and let $f^*:\NS(X)_\RR\to\NS(X)_\RR$ be the
induced map on the N\'eron--Severi group
$\NS(X)_\RR=\NS(X)\otimes\RR$. Further let~$\rho(T,V)$ denote the
spectral radius of a linear transformation~$T:V\to V$. Then
the (\emph{first}) \emph{dynamical degree of~$f$} is the quantity
\[ 
  \d_f = \lim_{n\to\infty} \rho\bigl((f^n)^*,\NS(X)_\RR\bigr)^{1/n}.
\]
Dynamical degrees have been much studied over the past couple of
decades; see~\cite{kawsilv:arithdegledyndeg} for a partial list of
references.
\par
In two earlier papers~\cite{kawsilv:arithdegledyndeg,arxiv1111.5664},
the authors studied an analogous arithmetic degree, which we now
describe.  Assume that~$X$ and~$f$ are defined over~$\Qbar$, and write
$X(\Qbar)_f$ for the set of points~$P$ whose forward $f$-orbit 
\[
  \Orbit_f(P) = \{P,f(P),f^2(P),\ldots\}
\]
is well-defined.  Further, let $h_X:X(\Qbar)\to[0,\infty)$ be a Weil
height on~$X$ relative to an ample divisor,
and let $\hplus_X=\max\{1,h_X\}$.  The \emph{arithmetic
degree of~$f$ at~$P\in X(\QQ)_f$} is the quantity
\begin{equation}
  \label{eqn:limdefafP}
  \a_f(P) = \lim_{n\to\infty} \hplus_X\bigl(f^n(P)\bigr)^{1/n},
\end{equation}
assuming that the limit exists. We also define \emph{upper and lower
arithmetic degrees} by the formulas
\[
  \aupper_f(P) = \limsup_{n\to\infty} \hplus_X\bigl(f^n(P)\bigr)^{1/n}
  \quad\text{and}\quad
  \alower_f(P) = \liminf_{n\to\infty} \hplus_X\bigl(f^n(P)\bigr)^{1/n}.
\]
It is proven in~\cite{kawsilv:arithdegledyndeg} that the values
of~$\aupper_f(P)$ and~$\alower_f(P)$ are independent of the choice
of the height function~$h_X$.
\par
A principal result of~\cite{kawsilv:arithdegledyndeg} is 
the fundamental inequality
\begin{equation}
  \label{eqn:fundineq}
  \aupper_f(P)\le \d_f\quad\text{for all $P\in X(\QQ)_f$.}
\end{equation}
The papers~\cite{kawsilv:arithdegledyndeg,arxiv1111.5664} also contain
a number of conjectures, which we recall here. The conjectures give
additional properties and relations for arithmetic and
dynamical degrees.

\begin{conjecture}
\label{conjecture:afPeqdfY}
Let $X/\Qbar$ be a \textup(normal\textup) projective variety, let
$f:X\dashrightarrow X$ be a dominant rational map defined
over~$\Qbar$, and let~$P\in X(\Qbar)_f$.
\begin{parts}
\Part{(a)}
The limit~\eqref{eqn:limdefafP} defining~$\a_f(P)$ exists.
\Part{(b)}
If~$\Orbit_f(P)$ is Zariski dense in~$X$, then
$\a_f(P)=\d_f$.
\Part{(c)}
$\a_f(P)$ is an algebraic integer.
\Part{(d)}
The collection of arithmetic degrees 
\[
  \bigl\{\a_f(Q):Q\in X(\Qbar)_f\bigr\}
\]
is a finite set.
\end{parts}
\end{conjecture}

In~\cite{kawsilv:arithdegledyndeg}, we stated without proof a number
of cases for which we could prove
Conjecture~\ref{conjecture:afPeqdfY}, and we promised that the proofs
would appear in a subsequent publication. This paper, which is that
publication, contains proofs of the following results.

\begin{theorem}
\label{theorem:casesconjistrue}
Conjecture~$\ref{conjecture:afPeqdfY}$ is true in the following
situations\textup:
\begin{parts}
\Part{(a)}
$f$ is a morphism and $\NS(X)_\RR=\RR$.
\Part{(b)}
$f$ is the extension to~$\PP^N$ of a regular affine
automorphism~$\AA^N\to\AA^N$.
\Part{(c)}
$X$ is a smooth projective surface and $f$ is an automorphism.
\Part{(d)}
$f:\PP^N\dashrightarrow\PP^N$ is a monomial map and we consider
only points~$P\in\GG_m^N(\Qbar)$.
\end{parts}
\end{theorem}
\begin{proof}
The proofs of~(a),~(b),~(c), and~(d) are given, respectively, in
Theorems~\ref{theorem:afPNSeqRnew},~\ref{theorem:afPdf1regaffaut},~\ref{theorem:surfaceautafPeqdf}, and~\ref{theorem:afPmonomial}.
\end{proof}

The following weaker result, which provides some additional evidence
for Conjecture~\ref{conjecture:afPeqdfY}, was also stated without
proof in~\cite{kawsilv:arithdegledyndeg}. The proof is given in this
paper.

\begin{theorem}
\label{theorem:mainthma}
Let $f:\AA^2\to\AA^2$ be an affine morphism defined over~$\Qbar$ whose
extension to $f:\PP^2\dashrightarrow\PP^2$ is dominant.
Assume that either of the following is true\textup{:}
\begin{parts}
\Part{(a)}
The map~$f$ is algebraically stable, i.e., for
all~$n\ge1$ we have  $(f^*)^n=(f^n)^*$ on~$\NS(\PP^2)_\RR$.
\Part{(b)}
$\deg(f)=2$.
\end{parts}
Then
\[
  \bigl\{P\in \AA^2(\Qbar) : \a_f(P)=\d_f\bigr\}
\]
contains a Zariski dense set of points having disjoint orbits.
\end{theorem}

The proof of Theorem~\ref{theorem:mainthma} uses $p$-adic methods,
weak lower canonical heights, and Guedj's classification of degree~$2$
planar maps~\cite{MR2097402}. The tools that we develop, specifically
Proposition~\ref{proposition:hfge0impafeqdf} and
Theorem~\ref{theorem:htwithperpt}, can be used to prove
Theorem~\ref{theorem:mainthma} more generally for affine morphisms
having a periodic point in the hyperplane at infinity.

\begin{remark}
We also note that Jonsson and Wulcan~\cite{arxiv1202.0203} have proven
a result on dynamical canonical heights that implies parts of
Conjecture~\ref{conjecture:afPeqdfY} for polynomial
morphisms~$\f:\AA^2\to\AA^2$ of small topological degree. These are
maps satisfying $\#f^{-1}(Q)<\d_f$ for a general
point~$Q\in\AA^2(\Qbar)$. Their proof uses a recent dynamical
compactification of~$\PP^2(\CC)$ due to Favre and
Jonsson~\cite{arxiv0711.2770}.
\end{remark}

\begin{remark}
We also mention the following related result
from~\cite{kawsilv:jordanblock}.  If $f:X\to X$ is a morphism and
$\Pic^0(X)_\RR=0$, then for all $P\in X(\Qbar)$, the limit defining
$\a_f(P)$ exists, and further the set $\bigl\{\a_f(Q):Q\in
X(\Qbar)\bigr\}$ is a finite set of algebraic integers. In other
words, under these hypotheses, Conjecture~1(a,c,d) is true; but we are
not able to prove Conjecture~1(b).
\end{remark}

\begin{acknowledgement}
The authors would like to thank ICERM for providing a stimulating
research environment during their spring 2012 visits.  The authors
would also like to thank the organizers of conferences on
Automorphisms (Shirahama 2011), Algebraic Dynamics (Berkeley 2012),
and the SzpiroFest (CUNY 2012), during which some of this research was
done.
\end{acknowledgement}

\section{Proof of Theorem~$\ref{theorem:casesconjistrue}$}
\label{section:prfofthma}

In this section we prove the various parts of
Theorem~\ref{theorem:casesconjistrue}, which give cases in which
Conjecture~$\ref{conjecture:afPeqdfY}$ is true. We remark that the
maps in Theorem~\ref{theorem:casesconjistrue}(a,b,c) are algebraically
stable, i.e., $(f^n)^*=(f^*)^n$.  This is automatic for morphisms, and
it is also a standard fact that it is true for regular affine
automorphisms.
Further, if~$f$ is algebraically stable, then
\[
  \d_f
   = \lim_{n\to\infty} \rho\bigl((f^n)^*\bigr)^{1/n}
   = \lim_{n\to\infty} \rho\bigl((f^*)^n\bigr)^{1/n}
   = \rho(f^*),
\]
so~$\d_f$ is automatically an algebraic integer.  Monomial maps are
not, in general, algebraically stable, but their dynamical degrees are
known to be algebraic integers~\cite{MR2358970}.  Thus in the proof of
Theorem~\ref{theorem:casesconjistrue}, if we prove that
$\a_f(P)=\d_f$, then we also know that $\a_f(P)$ is an algebraic
integer.

\subsection{Proof of Theorem~\ref{theorem:casesconjistrue}(a)}

We start with a result that is somewhat more general than
Theorem~\ref{theorem:casesconjistrue}(a).

\begin{theorem}
\label{theorem:afPNSeqRnew}
Let $X/\Qbar$ be a \textup(normal\textup) projective variety,
let~$f:X\to X$ be a morphism defined over~$\Qbar$, and suppose that
there is an ample divisor class~$D\in\NS(X)_\RR$ satisfying
\[
  f^*D \equiv \d_f D.
\]
Let~$P\in X(\Qbar)$. Then
\[
  \a_f(P) = 
  \begin{cases}
    1&\text{if $P$ is preperiodic,}\\
    \d_f&\text{if $P$ is wandering.}\\
  \end{cases}
\]
\end{theorem}
\begin{proof}
If~$P$ is preperiodic, then directly from the definition we see
that~$\a_f(P)=1$.  Next, if $\d_f=1$, then the fundamental
inequality~$\aupper_f(P)\le\d_f$ from~\eqref{eqn:fundineq} gives
\[
  1 \le \alower_f(P) \le \aupper_f(P) \le \d_f = 1.
\]
Hence~$\a_f(P)$ exists and is equal to both~$\d_f$ and~$1$.
\par
We assume now that $P$ is not preperiodic and that~$\d_f>1$.
The fact that~$\d_f>1$ and $f^*D\equiv\d_fD$ means that we are in the
situation to apply the canonical height~$\hhat_{D,f}$ described
in~\cite[Theorem~5]{kawsilv:arithdegledyndeg}.
Since we have assumed that the divisor~$D$ is ample and that~$P$ is
not preperiodic, we see
from~\cite[Theorem~5(d)]{kawsilv:arithdegledyndeg}
that~$\hhat_{D,f}(P)\ne0$. 
Then~\cite[Theorem~5(c)]{kawsilv:arithdegledyndeg}
tells us that $\alower_f(P)\ge\d_f$. But~\eqref{eqn:fundineq} 
says that~$\aupper_f(P)\le\d_f$, which shows that the limit
defining~$\a_f(P)$ exists and satisfies $\a_f(P)=\d_f$.
\end{proof}

We next use Theorem~\ref{theorem:afPNSeqRnew} to prove
Theorem~\ref{theorem:casesconjistrue}(a).

\begin{proof}[Proof of Theorem~\textup{\ref{theorem:casesconjistrue}(a)}]
Let~$D$ be an ample divisor on~$X$. The assumption
that~$\NS(X)_\RR=\RR$ implies that $f^*D\equiv dD+T$ for
some~$d\in\RR$ and~$T\in\NS(X)_\tors$.  Replacing~$D$ by a multiple,
we may assume that~$T=0$, so~$f^*D\equiv dD$. Since~$f$ is a morphism,
we have $(f^n)^*D\equiv(f^*)^nD\equiv d^nD$, so~$d=\d_f$.  We are thus
in exactly the situation to apply Theorem~\ref{theorem:afPNSeqRnew}.
We conclude that~$\a_f(P)=1$ or~$\d_f$, respectively, depending on
whether~$P$ is or is not preperiodic.
\end{proof}

\begin{remark}
\label{remark:afPPiceqNS}
We mention that for a variety such as~$\PP^N$, which has
$\Pic(X)_\RR=\NS(X)_\RR=\RR$, Theorem~\ref{theorem:afPNSeqRnew} is an
immediate consequence of the classical theory of canonical heights for
polarized dynamical systems; see for
example~\cite{callsilv:htonvariety,MR2316407}. Thus~$f^*D$ is linearly
equivalent to~$dD$ with $d=\d_f$, and the associated canonical
height~$\hhat_{D,f}$ satisfies
\[
  h\bigl(f^n(P)\bigr) = \hhat_{D,f}\bigl(f^n(P)\bigr)+O(1)
  = d^n\hhat_{D,f}(P)+O(1).
\]
Hence
\[
  \a_f(P)
  = \lim_{n\to\infty} \hplus\bigl(f^n(P)\bigr)^{1/n}
  = \begin{cases}
     d&\text{if $\hhat_{D,f}(P)>0$,}\\
     1&\text{if $\hhat_{D,f}(P)=0$.}\\
  \end{cases}
\]
This completes the proof, since~$\hhat_{D,f}(P)>0$ if~$P$ is wandering
and $\hhat_{D,f}(P)=0$ if~$P$ is preperiodic.
\end{remark}

\subsection{Proof of Theorem~\ref{theorem:casesconjistrue}(b)}

The next result on regular affine automorphisms implies
Theorem~\ref{theorem:casesconjistrue}(b). 

\begin{definition}
Let~$f:X\dashrightarrow X$ be a rational map.  The \emph{indeterminacy
  locus of~$f$}, which we denote~$I_f$, is the subvariety of~$X$ on
which~$f$ is not well-defined.
\end{definition}

\begin{definition}
Let~$f:\AA^N\to\AA^N$ by an automorphism.  By abuse of notation, we
write~$f$ and~$f^{-1}$ also for the extensions of~$f$ and~$f^{-1}$ to
rational maps~$\PP^N\dashrightarrow\PP^N$, and we write~$I_f$
and~$I_{f^{-1}}$ for their indeterminacy loci in~$\PP^N$.  The map~$f$
is a \emph{regular affine automorphism} if $I_f\cap
I_{f^{-1}}=\emptyset$.
\end{definition}

\begin{theorem}
\label{theorem:afPdf1regaffaut}
Let $f:\AA^N\to\AA^N$ be a regular affine automorphism of degree
$d\ge2$ defined over~$\Qbar$, and let~$g$ denotes the restriction
of~$f$ to $\PP^N\setminus\AA^N$.  Then
\[
  \a_f(P) = \begin{cases}
     1&\text{if $P$ is periodic,}\\
     \d_f&\text{if $P\in\AA^N(\Qbar)$ is wandering,}\\
     \d_g&\text{if $P\in(\PP^N\setminus\AA^N)(\Qbar)_f$ is wandering.}\\
  \end{cases}
\]
\end{theorem}
\begin{proof}
If~$P$ is periodic, it is clear from the definition that~$\a_f(P)=1$.
We assume henceforth that~$P$ is wandering.
\par
If~$P\in\AA^N(\Qbar)$, the proof is similar to the proof sketched in
Remark~\ref{remark:afPPiceqNS}, using the theory of canonical heights
for regular affine automorphisms developed by the first author. It is
proven in~\cite{arxiv0909.3573} that for all $Q\in\AA^N(\Qbar)$, the
limit
\[
  \hhatplus(P) = \lim_{n\to\infty}\frac{1}{d^n}h\bigl(f^n(P)\bigr)
\]
exists and satisfies
\[
  \hhatplus(Q)=0 
  \quad\Longleftrightarrow\quad
  \text{$Q$ is periodic.}
\]
Since~$P$ is assumed wandering, we have~$\hhatplus(P)>0$.  Choose
an~$n_0$ so that $h\bigl(f^n(P)\bigr)\ge(d^n/2)\hhatplus(P)$ for all
$n\ge n_0$. Then
\[
  \alower_f(P)
   = \liminf_{n\to\infty} \hplus\bigl(f^n(P)\bigr)^{1/n}
  \ge  \liminf_{n\to\infty} 
         \left(\frac{d^n}{2}\hhatplus(P)\right)^{1/n}
  = d.
\]
Hence~$\alower_f(P)\ge\d_f$, and combined with~\eqref{eqn:fundineq},
we deduce as usual that $\a_f(P)$ exists and equals~$\d_f$.
\par 
It remains to deal with wandering points
in~$\PP^N(\Qbar)_f\setminus\AA^N(\Qbar)$, i.e., points~$P$ lying on
the hyperplane at infinity.  To ease notation, we let
\[
  X = I_{f^{-1}} \subset \PP^N\setminus\AA^N
\]
be the indeterminacy locus of~$f^{-1}$. A theorem of
Sibony~\cite[Proposition~2.5.3]{sibony:panoramas} says that the
restriction of~$f$ to~$X$ is a surjective morphism
\[
  g = f|_X : X \to X.
\]
We claim that for~$P\in X(\Qbar)$, the map~$g$ satisfies
\begin{equation}
  \label{eqn:agPdg1wandper}
  \a_g(P) = \begin{cases}
     \d_g&\text{if $P$ is wandering,}\\
     1&\text{if $P$ is preperiodic.}\\
  \end{cases}
\end{equation}
We will verify this claim by using Theorem~\ref{theorem:casesconjistrue}(a).
\par
Let~$\ell=\dim X$. By~\cite[Proposition~2.5.4]{sibony:panoramas} there
exists an equidimensional surjective morphism $\pi:\PP^\ell\to X$.
Let~$\tilde X$ be the normalization of~$X$. Since~$\PP^\ell$ is
normal, the map~$\pi$ lifts to a surjective morphism
$\tilde\pi:\PP^\ell\to\tilde X$. We are going to show that
\[
  \tilde\pi^* : \NS(\tilde X)\longrightarrow\NS(\PP^\ell)
\]
is injective. 
\par
Let $D\in\NS(\tilde X)$ with $\tilde\pi^*(D)=0\in\NS(\PP^\ell)$. 
Let~$C$ be a curve on~$\tilde X$ and choose a curve~$\tilde
C\subset\PP^\ell$ with $\tilde\pi_*(\tilde C)=C$. Then
\[
  0 = \tilde\pi^*(D)\cdot\tilde C = D\cdot\tilde\pi_*(\tilde C)
  = D\cdot C.
\]
This is true for every curve~$C$ on~$\tilde X$, and
hence~$D=0\in\NS(\tilde X)$, which proves the injectivity
of~$\tilde\pi^*$. Since~$\NS(\PP^\ell)=\ZZ$, we conclude
that~$\NS(\tilde X)=\ZZ$.
\par
Let $\tilde g:\tilde X\to\tilde X$ denote the morphism induced
by~$g:X\to X$, and let $p:\tilde X\to X$ denote the normalization
morphism. Chooose any ample divisor~$D$ on~$X$.  The fact
that~$\NS(\tilde X)=\ZZ$ implies that~$p^*D$ is ample on~$\tilde X$.
Let~$P\in X(\Qbar)$ and choose a point~$Q\in p^{-1}(P)\subset\tilde
X(\Qbar)$. Then
\[
  h_{p^*D}\bigl(\tilde g^n(Q)\bigr)
  = h_D\bigl(p\circ\tilde g^n(Q)\bigr)+O(1)
  = h_D\bigl(g^n(P)\bigr)+O(1),
\]
wehre the~$O(1)$ is independent of~$n$,~$P$, and~$Q$. It follows directly
from the definition that $\a_{\tilde g}(Q)$ exists if and only if $\a_g(P)$
exists, and if they exist, then they are equal. It is also clear from
the geometry that~$\d_g=\d_{\tilde g}$; and since~$\tilde g$ is a morphism,
its dynamical degree is equal to~$\rho(\tilde g^*)$, which is
an algebraic integer.
\par
Since~$p^*D$ is ample, the inverse image~$p^{-1}(P)$ is a finite set,
so~$P$ is preperiodic if and only if~$Q$ is preperiodic.  The
assertion~\eqref{eqn:agPdg1wandper} now follows from
Theorem~\ref{theorem:casesconjistrue}(a).
\par
We now resume the proof of Theorem~\ref{theorem:afPdf1regaffaut},
where we recall that we are reduced to the case that
$P\in\PP_f^N\setminus\AA^N$. From~\cite[Proposition~2.5.3]{sibony:panoramas}
we have
\begin{equation}
  \label{eqn:fPNANIfsubIf1}
  f\bigl(\PP^N\setminus(\AA^N\cup I_f)\bigr) = I_{f^{-1}},
\end{equation}
so the fact that~$I_{f^{-1}}\cap I_f=\emptyset$ (which is the definition
of regularity) implies that if~$Q\notin I_f$, then the entire forward 
orbit of~$Q$ is well-defined, i.e., 
\[
  \PP^N(\Qbar)_f = \PP^N(\Qbar)\setminus I_f.
\]
In any case, from~\eqref{eqn:fPNANIfsubIf1}
we see that our wandering point~$P$ satisfies $f(P)\in I_{f^{-1}}=X$, 
so the assertion~\eqref{eqn:agPdg1wandper} 
gives
\[
  \a_g\bigl(f(P)\bigr) = \d_g.
\]
Since~$g=f|_X$ and since we can compute the arithmetic degree using
the height associated to any ample divisor, we can compute~$\a_g$ using
a very ample height on~$X$ that is the restriction of a very ample height
on~$\PP^N$. Thus 
\[
  \a_g\bigl(f(P)\bigr)=\a_f\bigl(f(P)\bigr) =\a_f(P),
\]
where the last equality is~\cite[Lemma~12]{kawsilv:arithdegledyndeg}.
Hene~$\a_f(P)=\d_g$.
\end{proof}

\subsection{Proof of Theorem~\ref{theorem:casesconjistrue}(c)}

In order to prove Theorem~\ref{theorem:casesconjistrue}(c), which
deals with automorphisms of smooth projective surfaces, we use the
following construction of the first author~\cite{kawtopent2005}, which
generalized the second author's construction on~K3
surfaces~\cite{silverman:K3heights}. We also refer the reader
to~\cite{arxiv1202.0203}, which gives results for surface maps of small
topological degree.

\begin{theorem}
\label{theorem:surfaceauthts}
\textup{(\cite{kawtopent2005})}
Let $X$ be a smooth projective surface defined over~$\Qbar$, and let
$f:X\to X$ be an automorphism with~$\d_f>1$.
\begin{parts}
\Part{(a)}
There are only finitely many $f$-periodic irreducible curves in~$X$.
Let~$E_f$ be the union of these curves.
\Part{(b)}
There are divisors~$D^+$ and~$D^-$ in~$\Div(X)_\RR$ and associated
canonical height functions~$\hhatplus$ and~$\hhatminus$ satisfying
\[
  \hhatpm = h_{D^\pm}+O(1)
  \quad\text{and}\quad
  \hhatpm\circ f^{\pm1}=\d_f\hhatpm.
\]
\Part{(c)}
$\hhatplus+\hhatminus$ is a Weil height for a divisor in $\Div(X)_\RR$
that is nef and big.
\Part{(d)}
$\hhatplus(P)\ge0$ and $\hhatminus(P)\ge0$ for all $P\in X(\Qbar)$.
\Part{(e)}
Let~$P\in (X\setminus E_f)(\Qbar)$. Then
\[
  \hhatplus(P)=0
  \quad\Longleftrightarrow\quad
  \hhatminus(P)=0
  \quad\Longleftrightarrow\quad
  \text{$P$ is periodic.}
\]
\end{parts}
\end{theorem}
\begin{proof}
(a) is \cite[Proposition~3.1]{kawtopent2005}.  The rest of
Theorem~\ref{theorem:surfaceauthts} is
\cite[Theorem~5.2]{kawtopent2005} (including the proof) and
\cite[Proposition~5.5]{kawtopent2005}.
\end{proof}

\begin{theorem}
\label{theorem:surfaceautafPeqdf}
Let $X$ be a smooth projective surface defined over~$\Qbar$, let
$f:X\to X$ be an automorphism, and let~$E_f$ be the union of the
$f$-periodic irreducible curves in~$X$ as in
Theorem~$\ref{theorem:surfaceauthts}$.  Then for all~$P\in X(\Qbar)$,
\[
  \a_f(P) = \begin{cases}
  1&\text{if $P$ is periodic or $P\in E_f$,} \\
  \d_f&\text{if $P$ is wandering and $P\notin E_f$.} \\
  \end{cases}
\]
\end{theorem}
\begin{proof}
If~$\d_f=1$, then using~\eqref{eqn:fundineq}, we have as usual
\[
  1 = \d_f \ge  \aupper_f(P) \ge \alower_f(P) \ge  1,
\]
so~$\a_f(P)=\d_f=1$. Similarly, if~$P$ is a periodic point, then
directly from the definition we have~$\a_f(P)=1$.
\par
We assume henceforth that~$\d_f>1$ and that~$P$ is not periodic.
If~$E_f$ is non-empty, let~$\f:E_f\to E_f$ denote the restriction
of~$f$ to~$E_f$. Writing~$E=\bigcup C_i$ as a finite union of
irreducible curves, there is an iterate~$\f^m$ such
that~$\f^m\in\Aut(C_i)$ for all~$i$.  Considering the three cases of
genus~$0$,~$1$, and greater than~$1$, we see that automorphisms of curves
have dynamical degree~$1$, so~$\d_{f^m}(C_i)=1$. It follows as above
that~$\a_{f^m}(P)=1$, since we can restrict an ample height on~$X$ to
each~$C_i$. Replacing~$P$ by~$f^i(P)$ for~$0\le i<m$, we deduce
that~$\a_f(P)=1$.
\par
We are now reduced to the case
that~$\d_f>1$,~$P\notin E_f$, and~$P$ is not periodic.
Let~$\hhatpm$ be the canonical heights associated to~$D^\pm$
as described in Theorem~\ref{theorem:surfaceauthts}.
In particular, we have
\[
  \hhatplus(P)>0,
  \quad
  \hhatplus\bigl(f^n(P)\bigr) = \d_f^n\hhatplus(P),
  \quad\text{and}\quad
  \hhatminus\bigl(f^n(P)\bigr) = \d_f^{-n}\hhatminus(P).
\]
We set~$h_X=\hhatplus+\hhatminus$, which is a Weil height associated to
a divisor that is big and nef. This allows us to compute
\begin{align*}
  \alower_f(P)
  &= \liminf_{n\to\infty} \hplus_X\bigl(f^n(P)\bigr)^{1/n} \\
  &= \liminf_{n\to\infty} \bigl(\hhatplus\bigl(f^n(P)\bigr)
           + \hhatminus\bigl(f^n(P)\bigr)\bigr)^{1/n} \\
  &= \liminf_{n\to\infty} \bigl(\d_f^n\hhatplus(P)
           + \d_f^{-n}\hhatminus(P)\bigr)^{1/n} \\
  &= \d_f,
\end{align*}
where to deduce the final equality, we are using the fact that
Theorem~\ref{theorem:surfaceauthts} tells us that~$\hhatplus(P)>0$.
\par
On the other hand, we know from~\eqref{eqn:fundineq} that the upper
arithmetic degree satisfies $\d_f\ge\aupper_f(P)$, so we have proven
that
\begin{equation}
  \label{eqn:alfPgedfgeaufP}
  \alower_f(P) \ge \d_f \ge \aupper_f(P) \ge \alower_f(P).
\end{equation}
Hence all of these quantities are equal, which proves that the
limit $\a_f(P)$ exists and is equal to~$\d_f$.
\end{proof}

\begin{remark}
Interesting cases to which Theorem~\ref{theorem:surfaceautafPeqdf}
applies are compositions of noncommuting involutions of~K3 surfaces
in~$\PP^2\times\PP^2$ and~$\PP^1\times\PP^1\times\PP^1$. The height
theory of these maps was studied in
in~\cite{MR1397435,baragarluijk06,silverman:K3heights,MR1352278}, and
Theorem~\ref{theorem:surfaceautafPeqdf} for K3~surfaces in
$\PP^2\times\PP^2$ was already proved by a similar argument
in~\cite[Section~12]{arxiv1111.5664}.  There are also higher
dimensional versions of these constructions in which the associated
involutions are rational maps, not morphisms.  It would be interesting
to study~$\a_f(P)$ for these reversible dynamical systems on
Calabi--Yau varieties.
\end{remark}

\subsection{Proof of Theorem~\ref{theorem:casesconjistrue}(d)}

The case of monomial maps described in
Theorem~\ref{theorem:casesconjistrue}(d) is an immediate consequence
of results in~\cite{arxiv1111.5664}.

\begin{theorem}
\label{theorem:afPmonomial}
Let~$A=(a_{ij})$ be an $N$-by-$N$ matrix with integer coefficients
and~$\det(A)\ne0$, and let~$f_A:\PP^N\dashrightarrow\PP^N$ be the
associated momonial map extending the endomorphism
of~$\GG_m^N\to\GG_m^N$ defined by~$A$, i.e., extending the map
\[
  (t_1,\ldots,t_N)\longmapsto
   \left(
    t_1^{a_{11}}t_2^{a_{12}}\cdots t_N^{a_{1N}},\ldots,
    t_1^{a_{N1}}t_2^{a_{N2}}\cdots t_N^{a_{NN}} \right).
\]
\begin{parts}
\Part{(a)}
The set of arithemtic degrees of~$f_A$ for points in~$\GG_m^N(\Qbar)$
satisfies
\[
  \bigl\{\a_{f_A}(P):P\in\GG_m^N(\Qbar)\bigr\}
  \subset \bigl\{\text{eigenvalues of $A$}\bigr\}.
\]
In particular, $\a_{f_A}(P)$ is an algebraic integer.
\Part{(b)}
If~$P\in\GG_m^N(\Qbar)$ has Zariski dense orbit, then~$\a_{f_A}(P)=\d_{f_A}$.
\Part{(c)}
If~$P\in\GG_m^N(\Qbar)$ satisfies~$\a_{f_A}(P)<\d_{f_A}$, then the
orbit of~$P$ lies in a proper $f_A$-invariant algebraic subgroup
of~$\GG_m^N$.
\end{parts}
\end{theorem}
\begin{proof}
This (and more) is proven in~\cite{arxiv1111.5664}.  In
particular,~(a) follows from~\cite[Corollary~32]{arxiv1111.5664},
and~(b) and~(c) follow by 
combining~\cite[Proposition~19(d) and Corollary~29]{arxiv1111.5664}.
\end{proof}

\section{Large Sets of Points Satisfying $\a_f(P)=\d_f$}
\label{section:densesets}

In this section we describe our main results concerning large sets for
which we can prove that $\a_f(P)=\d_f$. The proofs are given in
subsequent sections. We begin with a standard definition.

\begin{definition}
Let $f:X\dashrightarrow X$ be a dominant rational map.  The map~$f$
is said to be \emph{algebraically stable} (in codimension~$1$) if
the induced maps on $\NS(X)_\RR$ satisfy $(f^*)^n=(f^n)^*$ for
all~$n\ge1$.
\end{definition}

We note that one consequence of algebraic stability is that~$\d_f$ is
simply the largest eigenvalue of~$f^*$. In particular, if~$X=\PP^N$
and~$f$ is algebraically stable, then~$\d_f=\deg(f)$.  


We recall that Theorem~\ref{theorem:mainthma} states that for certain
affine morphisms $f:\AA^2\to\AA^2$, there is a large set of points~$P$
such that~$\a_f(P)=\d_f$.  Our proof of Theorem~\ref{theorem:mainthma}
actually yields a stronger result, which we now describe.

\begin{definition}
Let~$f:X\dashrightarrow X$ be a rational map defined over~$\Qbar$ with
dynamical degree~$\d_f>1$, and let~$D\in\Div(X)_\RR$.  The \emph{weak
  lower canonical height} associated to~$f$ and~$D$ is the function
\[
  \hhatlower_{f,D}^\circ : X(\Qbar)_f\longrightarrow\RR\cup\{\infty\},\qquad
  \hhatlower_{f,D}^\circ(P) 
    = \liminf_{n\to\infty} \frac{h_D\bigl(f^n(P)\bigr)}{\d_f^n}.
\]
Here~$h_D$ is any Weil height associated to~$D$.  Since any two such
heights differ by~$O(1)$, we see that the value
of~$\hhatlower_{f,D}^\circ(P)$ is independent of the choice of~$h_D$.
We also note that~$\infty$ is an allowable value
of~$\hhatlower_f^\circ$.
\end{definition}

\begin{remark}
The canonical height associated to eigendivisors of morphisms was
defined in~\cite{callsilv:htonvariety}. A more general definition for
rational maps $f:\PP^N\dashrightarrow\PP^N$, described by the second
author in~\cite{arxiv1111.5664}, is
\[
  \hhat_f(P) =
  \limsup_{n\to\infty} \frac{h\bigl(f^n(P)\bigr)}{n^\ell\d_f^n},
\]
where~$\ell\ge0$ is determined by the conjectural estimate
$\deg(f^n)\approx n^\ell\d_f^n$ as~$n\to\infty$.  The
function~$\hhatlower_f^\circ$ differs from~$\hhat_f$ in two ways.
First, it is defined using the liminf, rather than the limsup.
Second, the denominator includes only~$\d_f^n$, it has no~$n^\ell$
correction factor. The utility of the weak lower canonical height in
studying arithmetic degrees is explained below in
Proposition~\ref{proposition:hfge0impafeqdf}.  See
also~\cite{arxiv1202.0203} for additional material on canonical
heights attached to dominant rational maps.
\end{remark}

For self-maps of~$\PP^N$, it is proven in~\cite{arxiv1111.5664} that
\begin{equation}
  \label{eqn:hfPgt0impafPeqdf}
  \hhat_f(P)>0\quad\Longrightarrow\quad\aupper_f(P)=\d_f,
\end{equation}
so the positivity of~$\hhat_f(P)$ is at least as strong as the
equality of the dynamical degree and the upper arithmetic degree. The
proof works, \emph{mutatis mutandis}, to show that
if~$\hhatlower_f^\circ(P)>0$, then~$\alower_f(P)\ge\d_f$, and combined
with~\eqref{eqn:fundineq}, this implies that~$\a_f(P)=\d_f$, as in the
following useful result.

\begin{proposition}
\label{proposition:hfge0impafeqdf}
Let $f:X\dashrightarrow X$ be a dominant rational map defined
over~$\Qbar$ with dynamical degree satisfying~$\d_f>1$,
let~$D\in\Div(X)_\RR$ be any divisor, and let~$P\in X(\Qbar)_f$. Then
\[
  \hhatlower_{f,D}^\circ(P)>0
  \quad\Longrightarrow\quad
  \a_f(P)=\d_f.
\]
In particular, the limit~\eqref{eqn:limdefafP} defining~$\a_f(P)$
converges.
\end{proposition}
\begin{proof}
The assumption that $\hhatlower_{f,D}^\circ(P)>0$ implies in
particular that~$P$ is a wandering point. Further, since by definition
the height~$\hhatlower_{f,D}^\circ(P)$ is the liminf of
$\d_f^{-n}h_D\bigl(f^n(P)\bigr)$, we can find an integer~$n_0$ so that
\[ 
   \d_f^{-n} h_D\bigl(f^n(P)\bigr) \ge \frac12 \hhatlower_{f,D}^\circ(P) > 0
   \quad\text{for all $n\ge n_0$.}
\]
It follows that
\[
  \liminf_{n\to\infty} h_D\bigl(f^n(P)\bigr)^{1/n} 
 \ge \liminf_{n\to\infty}
    \d_f\left(\frac{1}{2} \hhatlower_{f,D}^\circ(P)\right)^{1/n} 
 = \d_f.
\]
\par
Let~$H\in\Div(X)$ be an ample divisor such that~$H-D$ is also ample. 
Then $h_{H}\ge h_D-C$ for a constant~$C$, so
\begin{align*}
  \alower_f(P)
  &= \liminf_{n\to\infty} \hplus_H\bigl(f^n(P)\bigr)^{1/n}\\*
  &\ge \liminf_{n\to\infty} \bigl(h_D\bigl(f^n(P)\bigr)-C\bigr)^{1/n} \\*
  &\ge \d_f.
\end{align*}
\par
This lower bound, combined with the upper bound~$\aupper_f(P)\le\d_f$
from~\eqref{eqn:fundineq}, implies that~$\a_f(P)$ exists and
equals~$\d_f$; cf.\ The final set~\eqref{eqn:alfPgedfgeaufP}
in the proof of Theorem~\ref{theorem:surfaceautafPeqdf}.
\end{proof}

\begin{question}
In the context of Proposition~\ref{proposition:hfge0impafeqdf},
if~$\a_f(P)=\d_f$, is it true that there exists a
divisor~$D\in\Div(X)_\RR$ such that $\hhatlower_{f,D}^\circ(P)>0$?
\end{question}

Using Proposition~\ref{proposition:hfge0impafeqdf}, we see that
Theorem~\ref{theorem:mainthma} is an immediate consequence of the
following result.

\begin{theorem}
\label{theorem:mainthm}
Let~$K/\QQ$ be a number field, and let $f:\AA^2\to\AA^2$ be an affine
morphism defined over~$K$ whose extension to
$f:\PP^2\dashrightarrow\PP^2$ is dominant and satisfies~$\d_f>1$.
Assume that one of the following is true\textup{:}
\begin{parts}
\Part{(a)}
The map~$f$ is algebraically stable.
\Part{(b)}
$\deg(f)=2$.
\end{parts}
Then there is a finite extension~$K'$ of~$K$, a prime~$\gp$ of~$K'$
and a $\gp$-adic neighborhood $U\subset\PP^2(K'_\gp)$ such that
\[
  \hhatlower_f^\circ(P) >0 \quad\text{for all $P\in U\cap\AA^2(K')$}.
\]
\end{theorem}

Before starting the proof of Theorem~\ref{theorem:mainthm}, we give
two lemmas. The first is the well-known characterization of
algebraic stabilty in the case of affine morphisms, and the second
describes how the dynamical degree and the canonical height change
when~$f$ is replaced by an iterate.

\begin{lemma}
\label{lemma:algstbiffVnnees}
Let $f:\AA^N\to\AA^N$ be an affine morphism, and by abuse of notation,
let $f:\PP^N\dashrightarrow\PP^N$ also denote the rational map obtained
by extending~$f$ to~$\PP^N$. Define inductively
a sequence of subvarieties of~$\PP^N$ by
\[
  V_0 = \PP^N\setminus\AA^N  \quad\text{and}\quad
  V_{n+1} = \overline{f(V_n\setminus I_f)},
\]
where the overline indicates to take the Zariski closure. Then
\[
  \text{$f$ is algebraically stable}
  \quad\Longleftrightarrow\quad
  \text{$V_n\ne\emptyset$ for all $n\ge0$.}
\]
\end{lemma}
\begin{proof}
This is well-known, but for the convenience of the reader, we give
the short proof. If~$V_n\ne\emptyset$, we write~$\xi_n$ for the
generic point of~$V_n$. Then~$\xi_n=f^n(\xi_0)$.
\par
Directly from the definition, we have that the map~$f$ is
algebraically stable if and only if $\deg(f^{n+1})=\deg(f^n)\deg(f)$
for all $n\ge1$. For polynomial maps~$f,g:\AA^N\to\AA^N$, we have in
general that
\[
  \deg(f\circ g)=\deg(f)\deg(g)
  \quad\Longleftrightarrow\quad
  g\bigl(\PP^N\setminus(\AA^N\cup I_g)\bigr) \not\subseteq I_f.
\]
Applying this with~$g=f^n$, we have
\begin{align*}
  \deg(f^{n+1})=\deg(f^n)\deg(f)
  &\quad\Longleftrightarrow\quad
    f^n(V_0\setminus I_{f^n}) \not\subseteq I_f \\
  &\quad\Longleftrightarrow\quad
    f^n(\xi_0)\notin I_f \\
  &\quad\Longleftrightarrow\quad
    V_n\not\subseteq I_f~\text{for all $n\ge0$, since the} \\
  &\omit\hfill  Zariski closure of $\xi_n=f^n(\xi_0)$ is $V_n$\\
  &\quad\Longleftrightarrow\quad
    V_n\ne\emptyset~\text{for all $n\ge0$.} 
\end{align*}
This conclude the proof of Lemma~\ref{lemma:algstbiffVnnees}.
\end{proof}

\begin{lemma}
\label{lemma:degsfmvsf}
Let $f:X\dashrightarrow X$ be a dominant rational map, and let~\text{$m\ge1$}.
\begin{parts}
\Part{(a)}
The dynamical degree satisfies
\[
  \d_{f^m}=\d_f^m.
\]
\Part{(b)}
Assume that~$X$ and~$f$ are defined over~$\Qbar$ and that~$\d_f>1$,
and let~$P\in X(\Qbar)_f$. Then
\[
  \hhatlower_f^\circ(P)
   = \min_{0\le i<m}  \d_f^{-i}\hhatlower_{f^m}^\circ\bigl(f^i(P)\bigr).
\]
\end{parts}
\end{lemma}
\begin{proof} 
(a) To ease notation, we write~$\rho(f^n)$ for the the magnitude of the
largest eigenvalue of~$(f^n)^*$ acting on~$\NS(X)_\RR$. Then by
definition~$\d_f$ is the limit of $\rho(f^n)^{1/n}$
as~$n\to\infty$. It is known that this limit exists, 
so any subsequence also converges to~$\d_f$. Hence
\[
  \d_f = \lim_{k\to\infty}\rho(f^{mk})^{1/mk}
   = \left(\lim_{k\to\infty}\rho\bigl((f^m)^k\bigr)^{1/k}\right)^{1/m}
   = \d_{f^m}^{1/m}.
\]
\par\noindent(b)\enspace
The lower canonical height is defined as a liminf, so we can't
restrict to a subsequence as in~(a).  Instead we apply the definition
of the canonical height to each of the
points~$P,f(P),\ldots,f^{m-1}(P)$. Thus
\begin{align*}
  \hhatlower_f^\circ(P)
  &= \liminf_{n\to\infty} \frac{h\bigl(f^{n}(P)\bigr)}{\d_f^{n}} \\*
  &= \liminf_{k\to\infty} \min_{0\le i<m} 
     \frac{h\bigl(f^{km+i}(P)\bigr)}{\d_f^{km+i}} \\
  &= \min_{0\le i<m} \liminf_{k\to\infty} 
     \frac{h\bigl((f^m)^k(f^iP)\bigr)}{\d_{f^m}^k\cdot\d_f^i} 
       \quad\text{from (a),}\\*
  &= \min_{0\le i<m}  \d_f^{-i}\hhatlower_{f^m}^\circ\bigl(f^i(P)\bigr),
\end{align*}
which completes the proof of Lemma~\ref{lemma:degsfmvsf}.
\end{proof}

 \section{Canonical heights, $p$-adic neighborhoods, and periodic points}
\label{section:htwithperpt}

The following result provides our primary tool for proving
Theorem~\ref{theorem:mainthm}.

\begin{theorem}
\label{theorem:htwithperpt}
Let~$K/\QQ$ be a number field, and let $f:\AA^N\to\AA^N$ be an affine
morphism defined over~$K$ whose extension 
$f:\PP^N\dashrightarrow\PP^N$ is dominant
and satisfies~$\d_f>1$. Suppose that there exists
an integer $m\ge1$ and a point~$Q_0\in\PP^N(K)$ lying on the
hyperplane at infinity such that~$f^m$ is defined at~$Q_0$ and such
that $f^m(Q_0)=Q_0$.  Then there is a prime~$\gp$ of~$K$ and a
$\gp$-adic neighborhood $U\subset\PP^N(K_\gp)$ of~$Q_0$ such that
\[
  \hhatlower_f^\circ(P) >0 \quad\text{for all $P\in U\cap\AA^N(K)$}.
\]
\end{theorem}

\begin{remark}
The study of algebraic points on varieties via $\gp$-adic
neighborhoods that are mapped into themselves by algebraic maps, as in
Theorem~\ref{theorem:htwithperpt}, has a long history. One might start
by citing the Skolem--Lech--Mahler theorem on linear recurrences and
Chaubauty's result on rational points on curves~\cite{MR0004484}, as
well as more recent results in arithmetic dynamics, including for
example papers on the dynamical Mordell--Lang
conjecture~\cite{arxiv:0808.3266,arXiv:0712.2344,arxiv0805.1560} and
applications to potential density~\cite{MR2862064}.
\end{remark}

\begin{example}
\label{example:fbadf2good}
We note that it is possible for~$f^m$ to be defined at~$Q_0$ even if
some lower iterate of~$f$ is not defined at~$Q_0$. For example,
the extension of the map
\[
  f:\AA^2\longrightarrow\AA^2,\quad
  f(x,y) = (y^2,x)
\]
to~$\PP^2$ has the form
\[
  f\bigl([X,Y,Z]\bigr)=[Y^2,XZ,Z^2].
\]
Its second iterate is
\[
  f^2\bigl([X,Y,Z]\bigr)=[X^2Z^2,Y^2Z^2,Z^4]=[X^2,Y^2,Z^2],
\]
so~$[1,0,0]\in I_f$, but $f^2\bigl([1,0,0]\bigr)=[1,0,0]$. For this
example, the map~$f^2$ is a morphism on~$\PP^2$, so
$\d_{f^2}=\deg(f^2)=2$, and Lemma~\ref{lemma:degsfmvsf}(a) tells
us that~$\d_f=\d_{f^2}^{1/2}=\sqrt2$.
\end{example}

\begin{remark}
\label{remark:fixedptalgst}
Under the assumptions of Theorem~\ref{theorem:htwithperpt}, the
map~$f^m$ is algebraically stable. To see this, we inductively define
subvarieties
\[
  V_0 = \PP^N\setminus\AA^N  \quad\text{and}\quad
  V_{n+1} = \overline{f^m(V_n\setminus I_{f^m})}.
\]
The assumption that~$f^m(Q_0)=Q_0$ implies that~$Q_0\in V_n$ for
all~$n$, and then Lemma~\ref{lemma:algstbiffVnnees} tells us
that~$f^m$ is algebraically stable. 
\end{remark}

\begin{example}
It is easy to construct examples of affine morphisms that do not
have any periodic points. For example, consider the map
\[
  f:\AA^2\longrightarrow\AA^2,
  \quad f(x,y)=(xy,y+1).
\]
Then
\[
  f^n(x,y) = \bigl(xy(y+1)\cdots(y+n-1),y+n\bigr)
\]
so~$f$ has no periodic points in~$\AA^2$. The extension of~$f$
to~$\PP^2$ satisfies
\[
  f^n\bigl([X,Y,Z]\bigr) =
  \bigl[XY(Y+Z)\cdots(Y+(n-1)Z),YZ^n+nZ^{n+1},Z^{n+1}\bigr],
\]
so the only possible periodic point of~$f$ in~$\PP^2\setminus\AA^2$ is
the point~$[1,0,0]$. But~$f$ is not defined at~$[1,0,0]$, and
hence~$f$ has no periodic points in~$\PP^2$.  Of course, the map~$f$
is not algebraically stable, since $\deg(f^n)=n+1$, so we already know
from Remark~\ref{remark:fixedptalgst} that
Theorem~\ref{theorem:htwithperpt} does not apply to~$f$.
\end{example}

\begin{example}
Let~$a\in\Qbar^*$ be a number that is not a root of unity.
The map
\[
  f:\PP^2\longrightarrow\PP^2,\quad
  f\bigl([X,Y,Z]\bigr) = [aX^2Y,XY^2,Z^3],
\]
which extends the affine morphism
\[
  \AA^2\longrightarrow\AA^2,\quad
  (x,y)\longmapsto(ax^2y,xy^2),
\]
gives an example of an algebraically stable affine morphism
having no periodic points on the line~$\{Z=0\}$ at infinity.  Indeed,
if we let
\[
  e(n)=\frac{1}{2}(3^n+1)
  \quad\text{and}\quad
  u(n)=\frac{1}{4}(3^n-1+2n),
\]
then one easily checks that
\[
  f^n\bigl([X,Y,Z]\bigr) 
  = \bigl[a^{u(n)}X^{e(n)}Y^{e(n)-1},a^{u(n)-n}X^{e(n)-1}Y^{e(n)},Z^{3^n}\bigr],
\]
so $\deg(f^n)=3^n=\deg(f)^n$, which shows that~$f$ is algebraically
stable. The indeterminacy locus of~$f$ is
\[
  I_f = \bigl\{[1,0,0],[0,1,0]\bigr\}.
\]
Suppose that~$[\a,\b,0]\notin I_f$ is a periodic point lying on the line
at infinity. Then~$\a\b\ne0$, so using the formula for~$f^n$, we must have
\[
  \frac{a^{u(n)}\a^{e(n)}\b^{e(n)-1}}{a^{u(n)-n}\a^{e(n)-1}\b^{e(n)}}
  = \frac{\a}{\b}.
\]
This implies that $a^{n}=1$, contradicting our choice of~$a$ as a
non-root of unity. Hence~$f$ has no periodic points on the line at
infinity. Of course, it does have periodic points in~$\AA^2$. More precisely,
it has exactly one periodic point, namely the fixed point~$(0,0)$.
\end{example}

\begin{proof}[Proof of Theorem~$\ref{theorem:htwithperpt}$]
Lemma~\ref{lemma:degsfmvsf}(b) says that
\[
  \hhatlower_f^\circ(P)
   = \min_{0\le i<m}  \d_f^{-i}\hhatlower_{f^m}^\circ\bigl(f^i(P)\bigr).
\]
So if we can prove the theorem for~$f^m$, then we can apply the
theorem to the map~$f^m$ and each of the
points~$P,f(P),f^2(P),\ldots,f^{m-1}(P)$ to deduce that the theorem is 
true for the map~$f$ and the point~$P$.  We may thus
replace~$f$ with~$f^m$, which reduces us
to the case that~$Q_0$ is a fixed point of~$f$. Then, as noted in
Remark~\ref{remark:fixedptalgst}, the map~$f$ is automatically
algebraically stable, i.e., $\d_f=d=\deg(f)$.  
\par
We let~$X_1,\ldots,X_N,W$ be projective coordinates on~$\PP^N$, with
the hyperplane at infinity being the set~$\{W=0\}$.  Making a change
of coordinates, we move~$Q_0$ to the point
\[
  Q_0 = [1,0,0,\ldots,0] \in \PP^N.
\]
Then the assumptions that~$f:\AA^N\to\AA^N$ and~$f(Q_0)=Q_0$ imply
that~$f$ can be written in the form
\begin{equation}
  \label{eqn:faX1dG1}
  f = [aX_1^d + G_1, G_2, \ldots, G_N, W^d]
\end{equation}
with~$a\in K^*$ and $G_1,\ldots,G_N\in K[X_1,\ldots,X_N,W]$ 
homogeneous polynomials of degree~$d$ that vanish at~$Q_0$,
i.e.,~$G_1,\ldots,G_N$ are in the ideal generated
by~$X_2,\ldots,X_N,W$.
\par
For any prime~$\gp$ of~$K$, we let
\[
  R_\gp=\{x\in K_\gp : |x|_\gp\le1\}
  \quad\text{and}\quad
  M_\gp=\{x\in K_\gp : |x|_\gp<1\}
\]
denote, respectively, the ring of integers of~$K_\gp$ and the maximal
ideal of~$R_\gp$.  We choose a prime~$\gp$ of~$K$ such that
\[
  a\in R_\gp^*
  \quad\text{and}\quad
  G_1,\ldots,G_N\in R_\gp[X,Y],
\]
and we consider the $\gp$-adic neighborhood of~$Q_0$ defined by
\[
  U = \bigl\{[x_1,x_2,\ldots,x_N,w] : 
     \text{$x_1\in R_\gp^*$ and  $x_2,\ldots,x_N,w\in M_\gp$} \bigr\}.
\]
Using~\eqref{eqn:faX1dG1} and the facts that~$a\in R_\gp^*$
and~$G_1,\ldots,G_N$ are in the ideal generated by~$X_2,\ldots,X_N$,
it is clear that~$f(U)\subset U$. More precisely, if we choose
a point
\[
  P = [\a_1,\ldots,\a_N,\b] \in U
  \quad\text{with $\a_1\in R_\gp^*$ and $\a_2,\ldots,\a_N,\b\in M_\gp$,}
\]
then we can write~$f^n(P)$ as
\begin{align*}
  f^n(P) = [\a_1^{(n)},\ldots,&\a_N^{(n)},\b^{d^n}] \in U\\*
   &\text{with $\a_1^{(n)}\in R_\gp^*$
         and $\a_2^{(n)},\ldots,\a_N^{(n)},\b\in M_\gp$.}
\end{align*}
The key point to note here is that we cannot cancel any factors
of~$\gp$ from these homogeneous coordinates of~$f^n(P)$, because the
first coordinate is a unit.
\par
We now compute
\begin{align*}
  h\bigl(f^n(P)\bigr)
  &= h\bigl([\a_1^{(n)},\ldots,\a_N^{(n)},\b^{d^n}]\bigr) \\*
  &= h\left(\left[
      \frac{\a_1^{(n)}}{\b^{d^n}},\ldots,
               \frac{\a_N^{(n)}}{\b^{d^n}},1\right]\right) \\
  &= \sum_{v\in M_K} \log\max\left\{
      \left\|\frac{\a_1^{(n)}}{\b^{d^n}}\right\|_v,\ldots,
      \left\|\frac{\a_N^{(n)}}{\b^{d^n}}\right\|_v,1\right\} \\
  &\ge \log\max\left\{
      \left\|\frac{\a_1^{(n)}}{\b^{d^n}}\right\|_\gp,\ldots,
      \left\|\frac{\a_N^{(n)}}{\b^{d^n}}\right\|_\gp,1\right\} \\
   &= \log \|\b\|_\gp^{-d^n} 
    \quad\text{since $\|\a_1^{(n)}\|_\gp=1$ and $\|\a_i^{(n)}\|_\gp\le1$
            for all $i$,}\\*
   &= d^n \log \|\b\|_\gp^{-1}.
\end{align*}
Note that this inequality holds for all~$n\ge0$.
We are given that~$\b\in M_\gp$, so $\log \|\b\|_\gp^{-1}>0$. Hence
\[
  \hhatlower_f^\circ(P)
  = \liminf_{n\to\infty} \frac{h\bigl(f^n(P)\bigr)}{\d_f^n}
  = \liminf_{n\to\infty} \frac{h\bigl(f^n(P)\bigr)}{d^n}
  \ge \log \|\b\|_\gp^{-1}
  > 0.
\]
This concludes the proof of Theorem~\ref{theorem:htwithperpt}.
\end{proof}

 \section{Algebraically stable affine maps on $\AA^2$}
\label{section:algstbleA2}

In this section we use Theorem~\ref{theorem:htwithperpt} to prove
Theorem~\ref{theorem:mainthm}(a). We use the assumed algebraically
stability of~$f$ and a case-by-case analysis to find the required
periodic point lying on the line at infinity.

\begin{proof}[Proof of Theorem~$\ref{theorem:mainthm}$\textup{(a)}]
If~$\d_f=1$, then~\eqref{eqn:fundineq} and the fact
that~$\alower_f(P)$ is always greater than or equal to~$1$ implies
the $\a_f(P)=1$. We may thus assume that~$\d_f>1$.  Let
$d=\deg(f)$. We write $f$ in homogeneous form as
\[
  f(X,Y,Z)=\bigl[ F(X,Y)+ZF_1(X,Y,Z), G(X,Y)+ZG_1(X,Y,Z), Z^d \bigr].
\]
Since~$f$ has degree~$d$, we see that at least one of~$F$ and~$G$ is
nonzero. Changing coordinates, we may assume that $F\ne0$.  
\par
Let $H=\gcd(F,G)\in K[X,Y]$, and write
\[
  F=HF_0\quad\text{and}\quad G=HG_0
  \quad\text{with}\quad \gcd(F_0,G_0)=1,
\]
so the map~$f$ has the form
\begin{multline*}
  f(X,Y,Z) = \bigl[ H(X,Y)F_0(X,Y)+ZF_1(X,Y,Z), \\*
      H(X,Y)G_0(X,Y)+ZG_1(X,Y,Z), Z^d \bigr].
\end{multline*}
We consider three subcases, depending on the degree of~$F_0$.

\CaseAlt{1}{$\deg(F_0)\ge2$} 
Since
\[
  \deg G_0 = \deg G - \deg H = d - \deg H = \deg F - \deg H = \deg F_0,
\]
we have a well-defined map
\[
  \f  = [F_0,G_0] : \PP^1\longrightarrow \PP^1
\]
of degree at least~$2$. Such a
map~$\f$ has infinitely many distinct periodic orbits
in~$\PP^1(\Qbar)$~\cite{beardon:gtm}, while there are only
finitely many points in~$\PP^1$ satisfying $H(X,Y)=0$. (If~$H$ is
constant, there will be no such points.) So after replacing~$K$ by a
finite extension, we can find a $\f$-periodic
point~$Q_0=[x_0,y_0]\in\PP^1(K)$, say of period~$m$, such that
\begin{equation}
  \label{eqn:HfiQ0ne0}
  H\bigl(\f^i(Q_0)\bigr)\ne0 \quad\text{for all $0\le i<m$.}
\end{equation}
By abuse of notation, we also write~$Q_0=[x_0,y_0,0]\in\PP^2$, using
the natural identification of~$\PP^1$ with the line $Z=0$ in~$\PP^2$.
We note that~\eqref{eqn:HfiQ0ne0} implies that
\begin{align*}
  f(Q_0) &= f\bigl([x_0,y_0,0]\bigr)  \\
  &= \bigl[H(x_0,y_0)F_0(x_0,y_0),H(x_0,y_0)G_0(x_0,y_0),0\bigr] \\
  &= \bigl[F_0(x_0,y_0),G_0(x_0,y_0),0\bigr]
\end{align*}
is well-defined, since~$F_0$ and~$G_0$ have no nontrivial common
roots, and more generally~\eqref{eqn:HfiQ0ne0}
ensures that~$f^i(Q_0)$ is well-defined for all~$i\ge0$. With the
identification $\PP^1=\{Z=0\}\subset\PP^2$, we have
\[
  f^i(Q_0) = \f^i(Q_0)\quad\text{for all $i\ge0$,}
\]
and hence the point~$Q_0\in\PP^2(K)$ is an $m$-periodic point
for~$f$. It follows from Theorem~\ref{theorem:htwithperpt} that there
is a prime~$\gp$ and a $\gp$-adic neighborhood $Q_0\in
U\subset\PP^2(K_\gp)$ such that~$\hhatlower_f^\circ(P)>0$ for all $P\in
U\cap\AA^2(K)$.

\CaseAlt{2}{$\deg(F_0)=0$}
Again by degree considerations we see that~$G_0$ is also constant,
so~$f$ has the form
\[
  f = \bigl[ \a H(X,Y)+ZF_1(X,Y,Z), 
      \b H(X,Y)+ZG_1(X,Y,Z), Z^d \bigr]
\]
for some $[\a,\b]\in\PP^1$. If~$\b\ne0$, we conjugate~$f$ by the 
map
\[
  \psi(X,Y,Z)=[Y,\b X-\a Y,Z],
\]
to obtain
\[
  f^\psi=\psi\circ f\circ\psi^{-1}
  =\bigl[\b H\circ\psi^{-1}+ZG_1\circ\psi^{-1},
   Z(\b F_1\circ\psi^{-1}-\a G_1\circ\psi^{-1}),Z^d].
\]
So in all cases, after possibly changing corodinates and relabeling, we
are reduced to studying maps of the form
\[
  f(X,Y,Z) =
  \bigl[  H(X,Y)+ZF_1(X,Y,Z), ZG_1(X,Y,Z), Z^d \bigr].
\]

\CaseAlt{2.a}{$H(1,0)=0$}
Then~$f$ is not defined at~$[1,0,0]$, i.e., $[1,0,0]\in I_f$, and
hence
\[
  \overline{f\bigl(\{Z=0\}\setminus I_f\bigr)}
   = \bigl\{[1,0,0]\bigr\} \subset I_f.
\]
It follows from Lemma~\ref{lemma:algstbiffVnnees} that~$f$ is not
algebraically stable.  In fact, already at the second iterate we have
$\deg(f^2)<\deg(f)^2$.

\CaseAlt{2.b}{$H(1,0)\ne0$}
Then~$f$ is defined at~$[1,0,0]$, and~$[1,0,0]$ is a fixed point
of~$f$, so we can take~$Q_0=[1,0,0]$ and~$m=1$ in
Theorem~\ref{theorem:htwithperpt} to obtain the desired conclusion.

\CaseAlt{3}{$\deg(F_0)=1$} 
In this case the map
\[
  \f  = [F_0,G_0] : \PP^1\longrightarrow \PP^1
\]
has degree~$1$. Thus~$\f$ is a linear fractional transformation, so
after a change of coordinates (of~$\PP^2$ that maps the line at
infinity to itself), the map~$\f$ may be put into one of the following
two forms:
\[
  \f(X,Y)=[aX,Y] \quad\text{or}\quad \f(X,Y)=[X+bY,Y]
  \quad\text{with $a,b\in\Qbar^*$.}
\]
We note that for any~$\g\in\PGL_3(\Qbar)$, we have
\[
  \d_{\g\circ f\circ\g^{-1}}=\d_f
  \quad\text{and}\quad
  \aupper_{\g\circ f\circ\g^{-1}}(P)=\aupper_f(\g P),
\] 
so it is permissible to make this change of coordinates.  We 
write
\[
  H(X,Y) = c_kX^kY^{d-1-k} + \cdots
  \quad\text{with $k\ge0$ and $c_k\ne0$.}
\]
\par
We note that unless~$a=1$, the map~$\f$ has either one or two
fixed points, and no other periodic points. If one of those
fixed points is not in~$I_f$, then we can apply
Theorem~\ref{theorem:htwithperpt} to conclude the proof. However, if
the fixed points are in~$I_f$, i.e.,  if~$H(X,Y)$ vanishes at the
fixed points, then~$f$ has no periodic points on the line at
infinty, so we cannot use Theorem~\ref{theorem:htwithperpt}. We give
an alternative argument that works in all cases.
\par
Let~$K$ be a number field containing the coefficients of the
polynomials defining~$f$, and let~$\gp$ be a nonarchimedean place such
that the nonzero coefficients of~$H,F_0,F_1,G_0,G_1$ have $\gp$-adic
absolute value~$1$. We consider the $\gp$-adic open set
\[
  U = \bigl\{ P=[x,y,z]\in\PP^2(K_\gp) : 
    \text{$|x|_\gp>|y|_\gp>|z|_\gp$ and 
    $|y|_\gp^d>|x|_\gp^{d-1}|z|_\gp$} \bigr\}.
\]
We note that a point $[\a,\b,1]\in\AA^2(K_\gp)$ is in~$U$ if and only if
\begin{equation}
  \label{eqn:ptisinUiff}
  |\a|_\gp>|\b|_\gp>1 \quad\text{and}\quad |\b|_\gp^d > |\a|_\gp^{d-1}.
\end{equation}
We are going to prove that
\[
  \aupper_f(P)=\d_f\quad\text{for all $P\in U\cap\AA^2(K)$.}
\]
\par
Let $P\in U\cap\AA^2(K_\gp)$ and write~$P$ and~$f(P)$ as
\[
  P=[\a,\b,1]\quad\text{and}\quad f(P)=[\a',\b',1].
\]
We claim that
\begin{equation}
  \label{eqn:claimfPU}
  f(P)\in U\quad\text{and}\quad |\b'|_\gp \ge |\b|_\gp^d.
\end{equation}
\par
The assumption that~$P\in U$ tells us that we can write
\begin{equation}
  \label{eqn:RbRSaR1S1RSd1}
   |\b|_\gp=R\quad\text{and}\quad |\a|_\gp=RS
  \quad\text{with $R>1$, $S>1$, and $R>S^{d-1}$.}
\end{equation}
We now verify~\eqref{eqn:claimfPU} for the two cases
for~$\f$.

\CaseAlt{3.a}{$\deg(F_0)\ge1$ and $\f=[aX,Y]$}
We estimate the size of the two terms in the first coordinate
of~$f(P)$ as
\[
  \bigl|F_0(\a,\b)H(\a,\b)\bigr|_\gp
  = |a\a|_\gp |c_k\a^k\b^{d-1-k}|_\gp 
  = |\a|_\gp^{k+1}|\b|_\gp^{d-1-k}
  = R^dS^{k+1} 
\]
and
\[
  \bigl|F_1(\a,\b,1)\bigr|_\gp
  \le \max_{0\le i\le d-1}  |\a|_\gp^i |\b|_\gp^{d-1-i}
  = \max_{0\le i\le d-1}  R^{d-1}S^i
  = (RS)^{d-1}.
\]
We know from~\eqref{eqn:RbRSaR1S1RSd1} that $R^d>(RS)^{d-1}$, so the 
ultrametric inequality gives
\begin{equation}
  \label{eqn:aprime}
  |\a'|_\gp
  = \bigl|F_0(\a,\b)H(\a,\b)+F_1(\a,\b,1)\bigr|_\gp
  = \bigl|F_0(\a,\b)H(\a,\b)\bigr|_\gp
  = R^dS^{k+1}.
\end{equation}
\par
Similarly, the second coordinate of~$f(P)$ has the two terms
\[
  \bigl|G_0(\a,\b)H(\a,\b)\bigr|_\gp
  = |\b|_\gp |c_k\a^k\b^{d-1-k}|_\gp 
  = |\a|_\gp^{k}|\b|_\gp^{d-k}
  = R^dS^k 
  \ge R^d
\]
and
\[
  \bigl|G_1(\a,\b,1)\bigr|_\gp
  \le \max_{0\le i\le d-1}  |\a|_\gp^i |\b|_\gp^{d-1-i}
  = \max_{0\le i\le d-1}  R^{d-1}S^i
  = (RS)^{d-1}.
\]
Again using ~$R^d>(RS)^{d-1}$ from~\eqref{eqn:RbRSaR1S1RSd1}, we have
\begin{equation}
  \label{eqn:bprime}
  |\b'|_\gp
  = \bigl|G_0(\a,\b)H(\a,\b)+G_1(\a,\b,1)\bigr|_\gp
  = \bigl|G_0(\a,\b)H(\a,\b)\bigr|_\gp
  = R^dS^k.
\end{equation}
\par
Using the formulas $|\a'|_\gp=R^dS^{k+1}$ and $|\b'|_\gp=R^dS^k$
from~\eqref{eqn:aprime} and~\eqref{eqn:bprime}, it is now easy
to verify the claims in~\eqref{eqn:claimfPU}. First, we check
that~$f(P)\in U$. We have
\[
  |\b'|_\gp = R^dS^k \ge R^d > 1
  \quad\text{and}\quad
  \frac{|\a'|_\gp}{|\b'|_\gp} = S > 1,
\]
and further
\[
  \frac{|\b'|_\gp^d}{|\a'|_\gp^{d-1}}
  = \frac{R^d}{S^{d-k-1}}
  \ge \frac{R^d}{S^{d-1}}
  > R^{d-1}
  > 1.
\]
(We have used the inequality $R>S^{d-1}$
from~\eqref{eqn:RbRSaR1S1RSd1}.) This shows that $f(P)$
satisfies~\eqref{eqn:ptisinUiff}, so~$f(P)\in U$. Finally, we have
\[
  |\b'|_\gp = R^dS^k \ge R^d = |\b|_\gp^d,
\]
which completes the proof of~\eqref{eqn:claimfPU}
in Case~3a.

\CaseAlt{3.b}{$\deg(F_0)\ge1$ and $\f=[X+b,Y]$}
The proof is similar, so we just quickly sketch. We have
\begin{align*}
  \bigl|F_0(\a,\b)H(\a,\b)\bigr|_\gp
  &= |\a+b|_\gp |c_k\a^k\b^{d-1-k}|_\gp 
  = |\a|_\gp^{k+1}|\b|_\gp^{d-1-k}
  = R^dS^{k+1}, \hidewidth \\
 \bigl|F_1(\a,\b,1)\bigr|_\gp
 & \le \max_{0\le i\le d-1}  |\a|_\gp^i |\b|_\gp^{d-1-i}
  = \max_{0\le i\le d-1}  R^{d-1}S^i
  = (RS)^{d-1}, \\
 \bigl|G_0(\a,\b)H(\a,\b)\bigr|_\gp
  &= |\b|_\gp |c_k\a^k\b^{d-1-k}|_\gp 
  = |\a|_\gp^{k}|\b|_\gp^{d-k}
  = R^dS^k 
  \ge R^d, \\
 \bigl|G_1(\a,\b,1)\bigr|_\gp
 & \le \max_{0\le i\le d-1}  |\a|_\gp^i |\b|_\gp^{d-1-i}
  = \max_{0\le i\le d-1}  R^{d-1}S^i
  = (RS)^{d-1}.
\end{align*}
These are the same estimates that we proved in Case~3a, so
the rest of the proof of Case~3a carries over verbatim. 
\par
We now resume the proof of Case~3.  We let~$P=[\a_0,\b_0,1]\in
U\cap\AA^2(K)$, and for~$n\ge0$ we write
\[
  f^n(P) = [\a_n,\b_n,1].
\]
Using~\eqref{eqn:claimfPU}, we see by induction that for
all~$n\ge0$ we have
\[
  f^n(P)\in U
  \quad\text{and}\quad
  |\b_n|_\gp = |\b_0|_\gp^{d^n}.
\]
Now the usual calculation gives
\begin{align*}
  h\bigl(f^n(P)\bigr)
  &= \frac{1}{[K:\QQ]} \sum_{v\in M_K} \log\max\bigl\{|\a_n|_v,|\b_n|_v,1\bigr\} \\
  &\ge \frac{1}{[K:\QQ]}  \log\max\bigl\{|\a_n|_\gp,|\b_n|_\gp,1\bigr\} \\
  &\ge \frac{1}{[K:\QQ]}  \log|\b_0|_\gp^{d^n} \\
  &= d^n \frac{\log|\b_0|_\gp}{[K:\QQ]} .
\end{align*}
Hence
\[
  \aupper_f(P) = \limsup \hplus_X\bigl(f^n(P)\bigr)^{1/n} \ge d.
\]
On the other hand, from~\eqref{eqn:fundineq} we have
$\aupper_f(P)\le \d_f\le d$, which completes the proof in Case~3 that
$\aupper_f(P)=\d_f$ for all $P\in U\cap\AA^2(K)$. (We remark that in
Case~3 it is easy to check that~$f$ is algebraically stable, so 
$\d_f=d$; but in any case, one always has the inequality
$\d_f\le d$, which is all that we need here.)
\par
Cases~1,~2, and~3 cover all of the possibilities for the map~$\f$,
which completes the proof of Theorem~\ref{theorem:mainthm}(a).
\end{proof}

\section{Degree $2$ affine maps on $\AA^2$}
\label{section:deg2onA2}

In this section we give the proof of Theorem~\ref{theorem:mainthm}(b),
which deals with degree~$2$ affine morphism of~$\AA^2$.  The proof is
a case-by-case analysis, using the classification of dominant
quadratic polynomial maps of $\AA^2(\CC)$ due to Guedj.  We note
that Guedj's proof works over any algebraically closed field of
characteristic~$0$, so in particular it is valid over~$\Qbar$. We also
note that it suffices to prove Theorem~\ref{theorem:mainthm}(b) for
maps that are not algebraically stable, since algebraically stable
maps of~$\AA^d$ are covered by Theorem~\ref{theorem:mainthm}(a).  It
is worth noting that some non-algebraically stable maps actually have
no periodic points on the line at infinity, so we cannot directly
apply Theorem~\ref{theorem:htwithperpt} to these maps. However, in
each case we are able to prove the desired growth
of~$\hplus_X\bigl(f^n(P)\bigr)$ in an appropriate~$\gp$-adic neighborhood,
leading to the desired conclusion.

\begin{remark}
As noted earlier, there are maps~$f:\AA^2\to\AA^2$ to which
Theorem~\ref{theorem:htwithperpt} does not apply. For example,
consider the map $f(x,y)=(y,xy)$. It is easy to see
that~$\deg(f^n)$ is the $(n+2)^{\text{nd}}$-Fibonacci number,
so~$\d_f=\frac{1+\sqrt5}{2}$. If Theorem~\ref{theorem:htwithperpt}
were valid for~$f$, then there would be an integer~$m\ge1$ and a
point~$Q_0$ on the line at infinity with the property
that~$f^m(Q_0)=Q_0$. It would follow from
Remark~\ref{remark:fixedptalgst} that~$f^m$ is algebraically stable,
so $\d_{f^m}=\deg(f^m)$, and then Lemma~\ref{lemma:degsfmvsf} would
say that $\d_f=\d_{f^m}^{1/m}=\deg(f^m)^{1/m}$ is the
$m^{\text{th}}$-root of an integer. This contradicts
$\d_f=\frac{1+\sqrt5}{2}$. Of course, for this easy example, one can
give an explicit formula for~$f^m$ and check directly that~$f^m$ has
no fixed points on the line at infinity. But our indirect argument is
more widely applicable.
\end{remark}

\begin{theorem}
\label{theorem:guedj}
\textup{(Guedj \cite{MR2097402})}
Let $f: \AA^2 \to \AA^2$ be a dominant quadratic polynomial map
defined over $\Qbar$. Then $f$ is conjugate by a $\Qbar$-linear
automorphism of $\AA^2$ to one of the maps described in
Table~$\ref{table:guedj}$.
\end{theorem}

\begin{table}[t]
\begin{center}
\renewcommand{\arraystretch}{1.15}
\begin{tabular}{|*{4}{@{\hspace{4pt}}c@{\hspace{4pt}}|}} \hline
Case & $f(x,y)$ & Conditions & $\d_f$ \\ \hline\hline
\textbf{1.1}
   & $(y+ c_1, xy + c_2)$  & $c_1, c_2 \in \Qbar$ & $\frac{1+\sqrt5}{2}$
\\ \hline
\textbf{1.2}
     & $(y+ c_1, y(y-ax) + by + c_2)$
     &\begin{tabular}{c}
          $a, b, c_1, c_2 \in \Qbar$\\ $(a, b) \neq (0,0)$\\
      \end{tabular} & $2$
\\ \hline
\textbf{2.1a}
    &$(a x+ c_1, x^2 + by + c_2)$
    &\begin{tabular}{c}
      $a, b, c_1, c_2 \in \Qbar$\\
      $a b \neq 0$\\
     \end{tabular} 
    & $1$
\\ \hline
\textbf{2.1b}
    &$(a x+ c_1, xy + c_2)$
    &\begin{tabular}{c}
       $a, c_1, c_2 \in \Qbar$ \\
       $a\neq 0$\\
     \end{tabular} 
    & $1$
\\ \hline
\textbf{2.2a}
    &$(f_1(x), f_2(x, y))$
    &\renewcommand{\arraystretch}{1}\begin{tabular}{c}
       $\deg(f_1) = 2$\\
       $\deg(f_2) = 2$\\
       $\deg_y (f_2) = 1$\\
     \end{tabular} 
    & $2$
\\ \hline
\textbf{2.2b}
    &$(f_1(x), f_2(x, y))$
    &\renewcommand{\arraystretch}{1}\begin{tabular}{c}
      $\deg(f_1) =1$ \\
      $ \deg(f_2) = 2$ \\
      $\deg_y (f_2) = 2$\\
     \end{tabular} 
    & $2$
\\ \hline
\textbf{2.2c}
    &$(y, f_2(x, y))$
    &\renewcommand{\arraystretch}{1}\begin{tabular}{c}
       $\deg(f_2) = 2$\\
       $\deg_x (f_2) = 2$\\
       $\deg_y (f_2) = 2$\\
     \end{tabular} 
    & $2$
\\ \hline
\textbf{2.2d}
    &$(x y+ c_1, x (x + ay) + bx + c_2)$
    &\begin{tabular}{c}
        $a, b, c_1, c_2 \in \Qbar$ \\
     \end{tabular} 
    & $2$
\\ \hline
\textbf{3.1}
    &$(y, x^2 + a y + c)$
    &\begin{tabular}{c}
       $a, c \in \Qbar$
     \end{tabular} 
    & $\sqrt{2}$
\\ \hline
\textbf{3.2}
    &$(a y +c_1, x (x-y) + c_2)$
    &\begin{tabular}{c}
       $a, c_1, c_2 \in \Qbar$ \\
       $a \neq 0$\\
     \end{tabular} 
    & $\frac{1+\sqrt5}{2}$
\\ \hline
\textbf{3.3}
    &$(a x^2 + b x + c_1 + y, x(y + \alpha x)+c_2)$
    &\begin{tabular}{c}
      $a, b, c_1, c_2, \alpha \in \Qbar$\\ $a \neq 0$\\
     \end{tabular} 
    & $2$
\\ \hline
\textbf{3.4}
    &$(xy + c_1, x(x + a y) + bx + c_2 + \alpha y)$
    &\begin{tabular}{c}
       $a, b,c_1, c_2, \alpha \in \Qbar$\\
       $\alpha \neq 0$\\
     \end{tabular} 
    & $2$
\\ \hline
\textbf{3.5}
    &$f$ is a  morphism of $\PP^2$ 
    & & $2$
\\ \hline
\end{tabular}
\end{center}
\caption{Guedj's classification \cite{MR2097402} of dominant quadratic
  maps $\AA^2\to\AA^2$}
\label{table:guedj}
\end{table}

\begin{proof}[Proof of Theorem~$\ref{theorem:mainthm}$\textup{(b)}]
For a given point $P = (x_0, y_0) \in \AA^2(K)$, we write
\[
  f^n(P) = (x_n, y_n).
\]
As in the proof of Theorem~\ref{theorem:htwithperpt},
our aim is to find a prime~$\gp$ of~$K$ and points~$P$ such that
\[
  \liminf_{n\to\infty} 
  \frac{\log\max\bigl\{|x_n|_\gp,|y_n|_\gp,1\bigr\}}{\d_f^n} > 0.
\]
For such~$P$ we have
\begin{align*}
  \hhatlower_f^\circ(P)
  &= \liminf_{n\to\infty} \frac{h\bigl(f^n(P)\bigr)}{\d_f^n} \\*
  &= \liminf_{n\to\infty} \frac{1}{\d_f^n[K:\QQ]}\sum_{w\in M_K} 
             \log\max\bigl\{|x_n|_\gp,|y_n|_\gp,1\bigr\} \\
  &\ge   \liminf_{n\to\infty} 
  \frac{\log\max\bigl\{|x_n|_\gp,|y_n|_\gp,1\bigr\}}{\d_f^n[K:\QQ]}  \\*
  &> 0,
\end{align*}
which suffices to prove the theorem. We fix a prime~$\gp$ such that
every nonzero coefficient of~$f$ has~$\gp$-adic norm equal to~$1$.
We consider each of the cases in Guedj's classification
(Table~\ref{table:guedj}).
\par
Theorem~\ref{theorem:mainthm}(a) covers all maps in Guedj's table that
are algebraically stable, i.e., maps satisfying~$\d_f=2$, while maps
with~$\d_f=1$ always have~$\a_f(P)=1$, by~\eqref{eqn:fundineq} and the
fact that~$\alower_f(P)\ge1$. We are thus reduced to studying maps
satisfying $1<\d_f<2$, which are Cases~1.1,~3.1, and~3.2 in Guedj's
classification.  We consider each in turn.

\Case{1.1}  
Take $P = (x_0, y_0) \in \AA^2(K)$ with $|x_0|_\gp = |y_0|_\gp > 1$, and 
to ease notation, let $R=|x_0|_\gp=|y_0|_\gp$.
For $n\geq 1$, an easy induction shows that
\[
  |x_n|_\gp =R^{F_{n+1}} \quad\text{and}\quad |y_n|_\gp =R^{F_{n+2}},
\]
where $F_n$ is the $n^{\text{th}}$ Fibonacci number.  Hence
\[
  \liminf_{n\to\infty} 
  \frac{\log\max\bigl\{|x_n|_\gp,|y_n|_\gp,1\bigr\}}{\d_f^n} 
  = \liminf_{n\to\infty} \frac{F_{n+2}\log R}{\d_f^n}  
  = \frac{\d_f^2\log R}{\sqrt5} > 0,
\]
since $\d_f=\frac12(1+\sqrt5)$ and $F_n=(\d_f^n-\d_f^{-n})/\sqrt5$.

\Case{3.1}  
Although~$f$ is not a morphism of~$\PP^2$, its second iterate
\[
 f^2 \bigl([X,Y,Z]\bigr) 
    = \bigl[X^2 + aYZ + cZ^2, a X^2 + Y^2 + a^2YZ + (a+1) cZ^2,Z^2 \bigr]
\]
extends to a morphism of $\PP^2$. Let~$\xi$ be a root of
$\xi^2-\xi+a=0$ and replace~$K$ with~$K(\xi)$. Then the
point~$[1,\xi,0]$ is a fixed point of~$f^2$ lying on the line at
infinity, so we can apply Theorem~\ref{theorem:htwithperpt} to obtain
the desired result.

\Case{3.2}  
Take $P = (x_0, y_0) \in \AA^2(K)$ with $1< |x_0|_\gp < |y_0|_\gp$,
and let $R=|y_0|_\gp$.  For $n \geq 1$, we claim that $|x_n|_\gp =
R^{F_{n+1}}$ and $|y_n| = R^{F_{n+2}}$, where $F_n$ is the $n$-th
Fibonacci number.  Indeed, by induction we find that
\begin{gather*}
  |x_{n+1}|_\gp = |a y_n + c|_\gp = |y_n|_\gp = R^{F_{n+2}} \\
  \noalign{\noindent and}
  |y_{n+1}|_\gp = |x_n (x_n - y_n) + c_2|_\gp = |x_n y_n|_\gp = R^{F_{n+1} +F_{n+2} }
  = R^{F_{n+3}}. 
\end{gather*}
Hence just as in Case~1.1 we have
\[
  \liminf_{n\to\infty} 
  \frac{\log\max\bigl\{|x_n|_\gp,|y_n|_\gp,1\bigr\}}{\d_f^n} 
  = \liminf_{n\to\infty} \frac{F_{n+2}\log R}{\d_f^n}  
  = \frac{\d_f^2\log R}{\sqrt5} > 0.
\]
The completes the proof of Theorem~\ref{theorem:mainthm}(b).
\end{proof}


\begin{thebibliography}{10}

\bibitem{MR2862064}
E.~Amerik, F.~Bogomolov, and M.~Rovinsky.
\newblock Remarks on endomorphisms and rational points.
\newblock {\em Compos. Math.}, 147(6):1819--1842, 2011.

\bibitem{MR1397435}
A.~Baragar.
\newblock Rational points on {$K3$} surfaces in {$\mathbb{P}\sp
  1\times\mathbb{P}\sp 1\times\mathbb{P}\sp 1$}.
\newblock {\em Math. Ann.}, 305(3):541--558, 1996.

\bibitem{baragarluijk06}
A.~Baragar and R.~van Luijk.
\newblock {$K3$} surfaces with {P}icard number three and canonical vector
  heights.
\newblock {\em Math. Comp.}, 76(259):1493--1498 (electronic), 2007.

\bibitem{beardon:gtm}
A.~F. Beardon.
\newblock {\em Iteration of {R}ational {F}unctions}, volume 132 of {\em
  Graduate Texts in Mathematics}.
\newblock Springer-Verlag, New York, 1991.

\bibitem{arxiv:0808.3266}
J.~P. Bell, D.~Ghioca, and T.~J. Tucker.
\newblock The dynamical {M}ordell--{L}ang problem for \'etale maps.
\newblock {\em Amer. J. Math.}, 132(6):1655--1675, 2010.

\bibitem{arXiv:0712.2344}
R.~L. Benedetto, D.~Ghioca, P.~Kurlberg, and T.~J. Tucker.
\newblock A case of the dynamical {M}ordell-{L}ang conjecture.
\newblock {\em Math. Ann.}, 352(1):1--26, 2012.

\bibitem{callsilv:htonvariety}
G.~S. Call and J.~H. Silverman.
\newblock Canonical heights on varieties with morphisms.
\newblock {\em Compositio Math.}, 89(2):163--205, 1993.

\bibitem{MR0004484}
C.~Chabauty.
\newblock Sur les points rationnels des courbes alg\'ebriques de genre
  sup\'erieur \`a l'unit\'e.
\newblock {\em C. R. Acad. Sci. Paris}, 212:882--885, 1941.

\bibitem{arxiv0711.2770}
C.~Favre and M.~Jonsson.
\newblock Dynamical compactifications of {${\mathbf{C}}^2$}.
\newblock {\em Ann. of Math. (2)}, 173(1):211--248, 2011.

\bibitem{arxiv0805.1560}
D.~Ghioca and T.~J. Tucker.
\newblock Periodic points, linearizing maps, and the dynamical {M}ordell-{L}ang
  problem.
\newblock {\em J. Number Theory}, 129(6):1392--1403, 2009.

\bibitem{MR2097402}
V.~Guedj.
\newblock Dynamics of quadratic polynomial mappings of {$\mathbb{C}^2$}.
\newblock {\em Michigan Math. J.}, 52(3):627--648, 2004.

\bibitem{MR2358970}
B.~Hasselblatt and J.~Propp.
\newblock {D}egree-growth of monomial maps.
\newblock {\em Ergodic Theory Dynam. Systems}, 27(5):1375--1397, 2007.
\newblock Corrigendum vol. 6, page 1999.

\bibitem{arxiv1202.0203}
M.~Jonsson and E.~Wulcan.
\newblock Canonical heights for plane polynomial maps of small topological
  degree, 2012.
\newblock Math. Res. Lett., to appear, \url{arXiv:1202.0203}.

\bibitem{kawtopent2005}
S.~Kawaguchi.
\newblock Projective surface automorphisms of positive topological entropy from
  an arithmetic viewpoint.
\newblock {\em Amer. J. Math.}, 130(1):159--186, 2008.

\bibitem{arxiv0909.3573}
S.~Kawaguchi.
\newblock Local and global canonical height functions for affine space regular
  automorphisms, 2009.
\newblock \emph{Algebra and Number Theory}, to appear. \url{arXiv:0909.3573}.

\bibitem{kawsilv:jordanblock}
S.~Kawaguchi and J.~H. Silverman.
\newblock Canonical heights for {J}ordan blocks and an application to the
  arithmetic degree of a morphism, 2012.
\newblock in preparation.

\bibitem{kawsilv:arithdegledyndeg}
S.~Kawaguchi and J.~H. Silverman.
\newblock On the dynamical and arithmetic degrees of rational self-maps of
  algebraic varieties, 2012.
\newblock \url{arXiv:1208.0815}.

\bibitem{sibony:panoramas}
N.~Sibony.
\newblock Dynamique des applications rationnelles de {$\mathbb{P}\sp k$}.
\newblock In {\em Dynamique et g\'eom\'etrie complexes (Lyon, 1997)}, volume~8
  of {\em Panor. Synth\`eses}, pages ix--x, xi--xii, 97--185. Soc. Math.
  France, Paris, 1999.

\bibitem{silverman:K3heights}
J.~H. Silverman.
\newblock Rational points on {K3} surfaces: a new canonical height.
\newblock {\em Invent. Math.}, 105(2):347--373, 1991.

\bibitem{MR2316407}
J.~H. Silverman.
\newblock {\em The {A}rithmetic of {D}ynamical {S}ystems}, volume 241 of {\em
  Graduate Texts in Mathematics}.
\newblock Springer, New York, 2007.

\bibitem{arxiv1111.5664}
J.~H. Silverman.
\newblock Dynamical degree, arithmetic entropy, and canonical heights for
  dominant rational self-maps of projective space, 2011.
\newblock \emph{Ergodic Th. and Dyn. Sys.}, to appear, \url{arXiv:1111.5664}.

\bibitem{MR1352278}
L.~Wang.
\newblock Rational points and canonical heights on {$K3$}-surfaces in
  {$\mathbb{P}\sp 1\times\mathbb{P}\sp 1\times\mathbb{P}\sp 1$}.
\newblock In {\em Recent Developments in the Inverse Galois Problem (Seattle,
  WA, 1993)}, volume 186 of {\em Contemp. Math.}, pages 273--289. Amer. Math.
  Soc., Providence, RI, 1995.

\end{thebibliography}
\def\cprime{$'$}

\end{document}